\documentclass{amsart}

\usepackage{xcolor, fullpage, eucal}

\newtheorem{theorem}{Theorem}
\newtheorem*{theorem*}{Theorem}
\newtheorem{lemma}[theorem]{Lemma}

\theoremstyle{remark}
\newtheorem{example}[theorem]{Example}
\newtheorem{prop}[theorem]{Proposition}
\newtheorem*{problem}{Problem}

\title[Regular graphs and symmetric functions]{Regularity in weighted
  graphs\\ a symmetric function approach} \author[Marni Mishna]{Marni
  Mishna\\Department of Mathematics\\Simon Fraser University}

\newcommand{\rsc}[2]{\left\langle#1, #2\right\rangle} 
\def\field{K}

\newcommand{\cS}{{\mathcal J}}
\newcommand{\cK}{{\mathcal K}}

\begin{document}
\begin{abstract} In this note we consider $k$-regular
  multigraphs, where the possible edge multiplicities are
  controlled. These structures are considered in a question recently
  posed by Brendan McKay. We express the generating functions using
  the scalar product of symmetric functions, and consequently give
  conditions on when the classes are D-finite.  We appeal to symmetric
  species results of Mendez to write the expressions in a systematic
  way.
\end{abstract}


\maketitle

\section{Introduction}
The asymptotic enumeration of regular graphs is a compelling topic
that has appeared in many forms in combinatorics over the past half
century. There are several approaches, and each has its own
 conditions, and results. In this note we revisit the
symmetric function approach, introduced by Goulden, Jackson and
Reilly~\cite{GoJaRe83}, generalized by Gessel~\cite{Gess90}, and
automated by Chyzak, Mishna and Salvy~\cite{ChMiSa05}. The goal of
this work is to give insight on a recent related problem posed by
McKay:

\begin{problem}[McKay~\cite{McKa14}]
  Let $\cS$, and $\cS^*$ be subsets of the non-negative integers, and let
  $\mathbf{d} = d(n) = (d_1,\dots, d_n)$ be a vector of non-negative
  integers. Let $M(n,\cS,\cS^*)$ be the number of symmetric matrices whose
  diagonal entries are drawn from $\cS^*$ and off-diagonal entries from
  $\cS$, whose row sums are $d_1,\dots, d_n$. As usual in graph theory,
  entries on the diagonal are counted twice. We are interested in the
  asymptotic value of $M(n,\cS,\cS^*)$ in the sparse case, where the row
  sums do not grow very quickly with $n$.
\end{problem}

We consider this problem in the case that the $d_i$ take on a finite
number of values. (We call this set of possible values $\cK$). We show
that the generating function of the sequence $M(n, \cS, \cS^*)$ is
D-finite under certain conditions.

A univariate generating function is D-finite if it satisfies a linear
differential equation with polynomial coefficients. This property has
considerable implications on asymptotic enumeration: D-finite
functions have restrictions on their asymptotic form; asymptotic
information is encoded in the differential equation that it satisfies;
they can be treated with a number of automated tools. Although we do
not compute asymptotic formulas here, the fact that the generating
functions are D-finite can be useful for precisely such a computation.

We cast this problem in graph theoretic language as follows in order
to state our results precisely.  Let~$\cS$ and~$\cK$ be sets of
positive integers, and suppose additionally that $\cK$ is
finite. Let~$\mathcal{G}_{\cS,\cK,n}$ be the set of well-labelled graphs
on~$n$ vertices where edge weights are from
$\cS\cup\{0\}$ and the sum of weights of the edges incident to any
given vertex is an element of~$\cK$.  The case when all $d_i=k$ is
the case of $k$-regular graphs. Here a graph is well-labelled if
the label set is $\{1, 2, \dots, n\}$, where $n$ is the
number of vertices. 

The following is our main result. It appears below as
Theorem~\ref{thm:Main}.
\begin{theorem*} Let $\cS$ and $\cK$ be finite sets of positive
  integers. Let $G_{\cS,\cK}(z)$ be the generating function for the
  class~$\mathcal{G}_{\cS,\cK}$ of well-labelled graphs where edge
  weights are from $\cS\cup\{0\}$ and the sum of weights of the edges
  incident to any given vertex is an element of~$\cK$.  Then,
\[G_{\cS,\cK}(z)=\sum_n |\mathcal{G}_{\cS,\cK,n}|\, z^n\] 
  is D-finite. Furthermore, the differential
  equation satisfied by~$G_{\cS,\cK}(z)$ is theoretically computable.
\end{theorem*} 
The condition of finiteness on $\cS$ is not necessary, but is a
consequence of the finiteness of $\cK$. We show how to relax the
condition of finiteness, once there is more notation developed. 

For example~$G_{\{1\},\{k\}}(z)$ is generating function for simple,
labelled, $k$-regular graphs, and $G_{\{1,2,\dots, k\},\{k\}}(z)$ is
the generating function for labelled $k$-regular multigraphs. The
D-finiteness of these generating functions for generic~$k$ was
conjectured by Goulden, Jackson and Reilly~\cite{GoJaRe83}, and was
proved by Gessel~\cite{Gess90}.  Our strategy proves the more general
result by using the work of Mendez~\cite{Mend93} to create the same
framework as of~\cite{Gess90}, to which the work of~\cite{ChMiSa05}
then applies. This amounts to building a symmetric function encoding
of the full class of graphs, and then performing a subseries
extraction to realize the degree restriction. These are all
theoretically effective, and hence the differential equations are
potentially computable.

As Gessel noted, the class of all regular graphs is \emph{not\/}
D-finite. Consequently, our results, up to finite union of classes are
most likely optimal.

\subsection{Contribution: Structure of graphs with controlled
edge  multiplicities}
McKay remarked in his problem ``The simplest non-trivial case is
$\cS^*= \{0\}$ and $\cS = \{0, 2, 3\}$''.  In our notation, $\cS^*= \{0\}$
corresponds to the criterion that the graphs have no loops, and our
$\cS$ are the same. We first consider $d_i=k$ for all $i$, and loosen this to
consider then $d_i\in \cK$ for finite $\cK$ and conclude with some
more general comments. We also address $\cS^*=\{0,1\}$.

Our principal contribution is a new formulation of his problem, and
the resulting proof of D-finiteness. We find our approach to be of
interest as it gives a new way to view McKay's problem, and also a new
example for the class of symmetric species. In this context,
generalizations to hypergraph variants are very natural.

The outline of the strategy has three main steps:
\begin{enumerate}
\item Find a series encoding a superset of graphs with edge weights
  from~$\cS$;
\item Rewrite this series as symmetric function using the power sum
  symmetric functions;
\item Extract the subseries of terms with the desired degree sequence
  via a scalar product operation. 
\end{enumerate}

In Section~\ref{sec:Graphs} we describe how to write the graph
generating functions using symmetric functions. The construction
generalizes our previous work~\cite{Mish07} in a straightforward way
using the multiassemblies studied by Mendez~\cite{Mend93}. The
construction uses the species theory formalism~\cite{BeLaLe98}, but we
leave the category theoretic details to previous sources to avoid a
rather substantial detour that is well described elsewhere. We remark,
however, this appears to be different than the description of graphs
as a species recently developed by Gainer-Dewar and
Gessel~\cite{GaGe14}. That said, our Proposition~\ref{thm:multinomial}
bears some resemblance to the formulas in Section~6 of~Henderson's
species generalization~\cite{Hend05}, which is the formalism upon which
their work is based. Perhaps this problem is a good entry point for
that theory.

The D-finiteness result follows quickly once we adapt Waring's
formula. This is explained in Section~\ref{sec:DFSymSeries}. We
conclude with some directions on how to weaken the conditions as
stated.

\section{Labelled graph generating functions}
\label{sec:Graphs}
We start with a systematic encoding of graph classes using symmetric
functions.  

\subsection{Simple graphs and $X$-generating functions}
Let $G$ be a simple graph with vertex set~$V(G)=\{x_1,\dots, x_n\}$,
and edge set $E(G)$. We associate to~$G$ the monomial $\pi(G)$ defined
\[\pi(G)=\prod_{\{x_i,x_j\}\in E(G)}
x_ix_j = x_1^{d_1}x_2^{d_2}\dots x_k^{d_k},\]
where $d_i$ is the degree of $x_i$.  Let $\mathbf{G}(X)$ be
the generating function of the set of all labelled simple graphs~$\mathcal{G}$,
each with a vertex set a subset of $X=\{x_1,x_2\dots,\}$:
\[
\mathbf{G}(X)=\sum_{G\in \mathcal{G}} \pi(G) = \prod_{i<j} \left(1+x_ix_j\right).
\]
(Remark, this is a superset of the well-labelled graphs). To see the
second equality, remark that every edge is either present once, or not
at all.  Similarly, if $\overline{\mathcal{G}}$ is the set of graphs
that permit multiple edges (but not loops),
\[
\overline{\mathbf{G}}=\sum_{G\in \overline{\mathcal{G}}} \pi(G) = \prod_{i<j} \left(\frac{1}{1-x_ix_j}\right),
\]
as every edge exists some non-negative integer number of times. 

Under this description, the set~$\mathcal{G}$ is a symmetric
species~\cite{Mend93}, and the series encoding $\mathbf{G}(X)$ is what
Mendez calls the associated $X$-generating function. Our strategy is
to determine the $X$-generating function for the class of graphs
of~$\mathcal{G}_{\cS}$ in which edge weights are incorporated, and
treated as multiplicities. (Hence, that the edge weights be positive
integers is essential.) It is straightforward to get an expression for
this, and then it is a mechanical manipulation to get a form we
desire.

The $X$-generating function for the class of labelled graphs with
graphs edge weights from the set of positive integers~$\cS$ is
\begin{equation}
\label{eq:graph}
\mathbf{G}_\cS(X)=\sum_{G\in \mathcal{G}_\cS} \pi(G) = \prod_{i<j} \left(1+\sum_{s\in \cS}(x_ix_j)^s\right).
\end{equation}
For example,
$\mathbf{G}_{\{1,2,3,\dots\}}(X)=\overline{\mathbf{G}}(X)$.
When the $X$ is clear, we might choose to not indicate the variable
set.

To determine the series for graph classes with loops it is sufficient to
change the product index from~$i<j$ to~$i\leq j$. If, rather, the
loops have weights from a different set, say $\cS^*$, we multiply
Eq.~\eqref{eq:graph} by the product
$\prod_{i} \left(1+\sum_{s\in \cS^*}(x_ix_i)^s\right)$. 

The number of well-labelled graphs in this class with given degree
sequence $d_1, \dots, d_n$ is simply the coefficient of the monomial
$x_1^{d_1}x_2^{d_2}\dots x_n^{d_n}$. For suitably defined graph classes,
and degree vectors, this is the value under investigation by McKay. In
standard generating function notation we write the coefficient as
\[
[x_1^{d_1}x_2^{d_2}\dots
x_n^{d_n}] \mathbf{G}_\cS.
\]
As is always the case for symmetric species, the $X$-generating functions are
symmetric functions. We access the coefficient using classic symmetric
function operations.

\subsection{Expressing $\mathbf{G}_\cS(\mathbf{x})$ using symmetric functions}
Since we can relabel any graph with a different set of labels and it
remains in the class, $\mathcal{G}_\cS$ is a symmetric class with
respect to the graph labels. We leverage this underlying symmetry to
rewrite the generating function in terms of symmetric functions.

Let $\lambda=(\lambda_1, \dots, \lambda_k)$ be a integer partition of
$n$, a fact which we denote by~$\lambda\vdash n$.  Let
$X=x_1,x_2,\dots$ be an infinite, but countable, variable set. Then the
symmetric function $m_\lambda(X)$ (or simply $m_\lambda$) is defined
\[
m_\lambda(X)=\sum x_{i_1}^{\lambda_1}x_{i_2}^{\lambda_2}\dots x_{i_k}^{\lambda_k}
\]
where the sum is over all $k$-tuples of distinct positive integers
$(i_1, i_2,\dots, i_k)$. This is the monomial symmetric function
indexed by $\lambda$. The set of the monomial symmetric functions form
a basis for a vector space of symmetric functions over $\mathbb{Q}$.
 
We express the classic elementary, complete,  and power sum
symmetric functions in the monomial basis as follows:
\[ e_n = m_{(1,1,\dots, 1)}, 
\quad h_n = \sum_{\lambda\vdash n}m_\lambda, 
 \quad p_n = m_{(n)}.
\]
Recall the
definition~$e_\lambda=e_{\lambda_1}e_{\lambda_2}\dots
e_{\lambda_k}$.
The set of the elementary symmetric functions indexed by partitions
also forms a basis for a vector space of symmetric functions.  This is
also true for the sets of $h_\lambda$ and $p_\lambda$ respectively,
which are similarly defined. We work in the ring of symmetric series
over~$X$:
\[
\Lambda((X)):= \mathbb{Q}[[p_1, p_2, p_3, \dots]].
\]
In particular we are interested in symmetric series~$\mathbf{R}(X)$ of the form
\[
\mathbf{R}(X)=\sum_{0<n}\sum_{\lambda\vdash n} c_\lambda p_\lambda(X).
\]

The symmetric function operation known as {\sl plethysm\/} is
essential to our solution. Given two symmetric functions~$u$ and~$v$,
the inner law defines the quantity~$u[v]$ by defining the following
rules, with $u, v, w\in\Lambda$ and~ $\alpha,\beta$ in~$\field$
\begin{equation*}
  (\alpha u+\beta v)[w]=\alpha u[w]+\beta v[w],
  \quad
  (uv)[w]=u[w]v[w],
\end{equation*}
and, most importantly, if $w=\sum_\lambda c_\lambda p_\lambda$ then
$p_n[w]=\sum_\lambda c_\lambda
p_{(n\lambda_1)}p_{(n\lambda_2)}\ldots$.
For example, we can deduce that $w[p_n]=p_n[w]$, and in particular that
$p_n[p_m]=p_{nm}$.  In a mnemonic way:
\begin{equation}\label{eq:pleth}
w[p_n]=w(p_{1n},p_{2n},\dots,p_{kn},\ldots)
\qquad\text{whenever}\qquad
w=w(p_1,p_2,\dots,p_k,\ldots).
\end{equation}

Let $\mathbf{H}=\sum_n h_n$ and $\mathbf{E}=\sum_n e_n$. Gessel noted
that $\mathbf{G}$ and $\overline{\mathbf{G}}$ can both be
expressed using plethysm:
\begin{eqnarray}
\mathbf{G}&=& \prod_{i<j} \left(1+x_ix_j\right)=\sum_n
e_n[e_2]=\mathbf{E}[e_2].\\
\overline{\mathbf{G}}&=& \prod_{i<j}
\left(\frac{1}{1-x_ix_j}\right) =\sum_n h_n[e_n]=\mathbf{H}[e_2].
\end{eqnarray}

Given Eq.~\eqref{eq:pleth}, often plethysm expressions are easier to
manipulate when the symmetric functions are written in the power sum
basis. We do this next in Section~\ref{sec:powersum}, and this is
followed by a discussion on how to derive the plethym expressions for
$\mathbf{G}_\cS(X)$ in~\ref{sec:SymSpecies}.

\subsection{Expressions in the power sum basis}
\label{sec:powersum}
We recall the following classic lemma as it guides our work.  It shows
how to express an infinite sum of $h_n$ as a function of power sum
symmetric functions.
\begin{lemma}[Waring formula]\label{thm:exp-log}
The following equations are true:
\[
\mathbf{H}=\sum_nh_n=\prod_{i<j} \frac{1}{1-x_i}=\exp\left(\sum_{0<k} p_k/k\right),
\] 
\[
\mathbf{E}=\sum_ne_n=\prod_{i<j} (1+x_i)=\exp\left(\sum_{0<k} (-1)^{k+1}p_k/k\right).
\] 
\end{lemma}
\begin{proof}
The proof is elementary series manipulations:
\begin{eqnarray*}
 \log \prod_{0<i} \frac{1}{1-x_i} 
= \sum_{0<i}\log \frac{1}{1-x_i}
= \sum_{0<i}\sum_{0<k} x_i^k/k 
= \sum_{0<k}\sum_{0<i} x_i^k/k 
= \sum_{0<k} p_k/k.
\end{eqnarray*}
\end{proof}
Indeed, the plethysms are easier to analyse given this form:
\begin{eqnarray*}
\overline{\mathbf{G}}&=&\mathbf{H}[e_2] 
= \exp\left(\sum_{0<k}\frac{1}{k}
  p_k\right)[e_2]\\
&=& \exp(\sum p_k[e_2]/k)\\
&=&\exp(\sum \frac{1}{k} p_k[p_1^2/2-p_2/2])\\ 
&=&\exp(\sum \frac{1}{2k} (p_k^2-p_{2k})).
\end{eqnarray*}

We can similarly express $\mathbf{E}[e_2]$ (and indeed
$\mathbf{H}[h_2]$, $\mathbf{H}[e_2]$).
\section{Labelled graphs as a symmetric species}
\label{sec:SymSpecies}
One of the innovations of species theory~\cite{Joya81,BeLaLe98}, is a
rigorous combinatorial interpretation of the plethysm operation in
terms of natural compositions of combinatorial
structures~\cite{Berg87}. Plethysm as an analytic analog to
composition has been well studied since P\'olya's composition
theorem. Asymmetric series of Labelle~\cite{Labe93} is also an
important relative to the $X$-generating functions that we seek.

The combinatorial understanding of the composition in the particular
case of~$\mathbf{H}$ and~$\mathbf{E}$ is formally developed by
Mendez~\cite{Mend93}, and we gave a direct interpretation in the case
of graphs, and other variants in~\cite{Mish07}. In
particular,~\cite{Mish07} contains is a description of graphs as
multisort, and ultimately symmetric, species. The interpretation is as
follows: a simple labelled graph is a set of edges. Edges are sets of
atomic structures.  Each atom is coloured a colour from the infinite
set $X=\{x_1, x_2\dots, \}$, and atoms of the same colour are
identified to form a vertex.  The combinatorial composition of a set
of edges is reflected in~$X$-generating function by a plethysm.

In the notation of classic species, the class of all labelled
multigraphs with loops allowed given by the multisort species
$E[E_2](X)$, and the associated cycle index series is
$\mathbf{H}[h_2]$. Here $E$ is the species of sets, and $E_2$ is the
species of a set of cardinality two. Thus, a graph is a set of edges,
but repetitions are not controlled in the multisort version. The
multiassembly construction of Mendez controls repetition of
elements. A \emph{multiassembly of type $\lambda$\/}, denoted
$M_\lambda$, is a multiset where the multiplicities of the elements
are prescribed by the parts of the partition $\lambda$. For example, a
multiassembly of type $\lambda=(1,1,\dots,1)$ is a usual set without
repetitions.  Example~3.7 and Proposition~3.9 of~\cite{Mend93}
describe how to get the $X$-generating function associated a composition
of a multiassembly and a some object. In particular, the
$X$-generating function of a multiassembly $M_\lambda(X)$ is
$m_\lambda(X)$, and the composition is realized by plethysm.

\begin{example}
  We view graphs as multiassemblies of edges, where edge
  multiplicities are given by~$\lambda$. The edges themselves are
  multiassemblies of vertices. For example, the species
  $M_{(3,2)}[M_{(1,1)}(X)]$ is the set of graphs with one edge of
  multiplicity two, and one edge of multiplicity three, under all
  possible positive integer labellings. There are two possible shapes,
  both under all possible valid labellings: The first graph is on
  three vertices, and is path. Each possible
  labelling contributes a monomial of the form~$x_i^5x_j^3x_k^2$ to
  the $X$-generating function. The
  second shape is a graph with four vertices, and
  two edges, and the graph is not connected. Each graph of this type
  contributes a monomial $x_i^2x_J^2x_k^3x_\ell^3$. The generating function is
  thus $m_{5,3,2}+m_{3, 3, 2, 2}$. This is precisely~$m_{(3,2)}[m_{(1,1)}]$.
\end{example}

The $X$-generating of $\mathbf{G}_\cS$ of $\mathcal{G}_\cS$
requires a sum over all possible partitions with parts from $\cS$.
Towards a more compact notation, define the symmetric function
$f_{\cS,n}$ as follows:
\begin{equation}
f_{\cS,n}= \sum_{\lambda\vdash n; \lambda_i\in \cS}m_\lambda.
\end{equation}
To unravel the definition, remark $f_{\{2,3\}, 6}=m_{2,2,2}+m_{3,3}$
and further note that~$f_{\{1\},n}=e_n$ and
$f_{\{1,\dots, n\},n}=h_n$.
We summarize these results in the following proposition. 
\begin{prop}
\label{thm:species} 
Fix~$\cS$, a nonempty integer partition, and $X=x_1, x_2,
\dots$, and infinite, but countable set of labels. 
\begin{enumerate}
\item The symmetric species of simple graphs $\mathcal{G}_\cS$ is
 isomorphic to $\cup_{\lambda\vdash n; \lambda_i\in \cS}
 M_\lambda[M_{(1,1)}(X)]$. 
\item  For any degree sequence~$\mathbf{d}=(d_1, \dots, d_n)$ with
  sum $D=d_1+\dots+d_n$,  the coefficient
\[
[x_1^{d_1}x_2^{d_2}\dots x_n^{d_n}] \sum_{\lambda \vdash D; \lambda_i\in
  \cS} m_\lambda [ e_2]
\]
is precisely the number of graphs in $\mathcal{G}_{\cS,\mathbb{N},n}$
with degree sequence $\mathbf{d}$. Equivalently, this is the
coefficient of $m_{\mathbf{d}}$ when expanded in the multinomial
basis of symmetric functions.\footnote{With $\mathbf{d}$ normalized to
  be decreasing if necessary.}
\end{enumerate}
\end{prop}

Similarly, the analogous class of graphs permitting
loops\footnote{Graphs with loops of weight 1 corresponds to
  $\cS^*=\{0,1\}$ in the original problem.}, denoted
${\mathcal{G}}^\circ_\cS$, is isomorphic as species to
\[\bigcup_{\lambda\vdash n; \lambda_i\in \cS}
M_\lambda[M_{(1,1)}(X)+M_{(2)}(X)],\] and the coefficient
\[
[x_1^{d_1}x_2^{d_2}\dots x_n^{d_n}] \sum_{\lambda \vdash D; \lambda_i\in
  \cS} m_\lambda [h_2]
\]
is precisely the number of graphs in ${\mathcal{G}}^\circ_{\cS,\mathbb{N},n}$
with degree sequence~$\mathbf{d}$.

\subsection{Series expressions}
\label{sec:series}
To prove the D-finiteness results, we need the $X$-generating
functions in a different format. We define
$\mathbf{F}_{\cS}=\sum_n f_{\cS,n}$ and hence
$\mathbf{F}_{\{1\}}=\mathbf{E}$ and
$\mathbf{F}_{\{1,2,\dots\}}=\mathbf{H}$. Now,
$\mathbf{F}_{\cS}= \prod_{0<i} ( 1+ \sum_{s\in\cS} x_i^s)$.  We
can generalize Lemma~\ref{thm:exp-log} to express $\mathbf{F}_{\cS}$
in the power sum symmetric function basis.  The proof follows from
very basic manipulations.

\begin{prop}
\label{thm:multinomial}
Let $\cS=\{j_1, \dots, j_{\ell}\}$ be a set of distinct positive integers. Then,
\[\mathbf{F}_{\cS}=
\sum_{0<n} f_{\cS, n}= \exp\left(\sum_{0<n} a_n p_n\right)
\]
where $a_n$ is the following sum taken over all compositions
$\alpha=(\alpha_1, \alpha_2, \dots)$ of $n$ such that each
part~$\alpha_i$ is contained in $\cS$: 
\begin{equation}\label{eq:an}
a_n=-\sum_\alpha
\frac{(-1)^{\operatorname{parts}(\alpha)}}{\operatorname{parts}(\alpha)}.
\end{equation}
Here $\operatorname{parts}(\alpha)$ is the number of parts in the
composition. 
\end{prop}

\begin{proof}First we note that as the elements of~$\cS$ are all
  positive, $a_n$ is well defined since the number of such
  compositions is finite. Next we apply the same log-exp expansion,
  and some very basic coefficient extraction formulas:
\begin{eqnarray*}
\log \sum_{0<n} f_{\cS,n} 
&=& \log \prod_{0<i} ( 1+ \sum_{s\in\cS}x_i^s) \\
&=&-\sum_{0<i}\log \frac{1}{1-\sum_{s\in\cS}(-x_i^s)}\\ 
&=& \sum_{0<i} \sum_{0<k} \frac{(-1)^{k+1}}{k}(\sum_{s\in\cS} x_i^s)^k\\
&=&\sum_{0<i} \sum_{0<n} a_n x_i^n\\
&=&\sum_{0<n} a_n p_n.
\end{eqnarray*}
\end{proof}
There are a few simplifications to note. If~$\cS=\{1\}$, then
$a_n=\frac{-1}{n}$, since there is only one term in the summation in
Equation~\eqref{eq:an}.  When $\cS=\{1, 2, \dots \}$,
$a_n=\frac{1}{n}$ as the sum is over all compositions, and we invoke a
M\"obius inversion argument. Further, when $\cS=\{s_1,\dots, s_\ell\}$
is finite, we can express this as follows:
\begin{equation}
\sum_{0<n} f_{\cS, n} = \exp\left( \sum_{0<k}\frac{ (-1)^{k+1}}{k}
  \sum_{i_1+i_2+\dots+i_\ell = k} \binom{k}{i_1~i_2~\dots
    ~i_\ell} p_{s_1i_1+s_2i_2+\dots+s_\ell i_\ell}
\right).
\end{equation}

\begin{theorem}
\label{thm:graphGF}
Let~$\cS=\{j_1, \dots, j_{\ell}\}$ be a finite set of~$\ell$ distinct
positive integers. Then~$\mathbf{G}_\cS$, the~$X$-generating
function for the symmetric species of labelled simple graphs with edge
weights from~$\cS$ satisfies:
\begin{eqnarray*}
\mathbf{G}_\cS&=& \left(\sum_n f_{\cS,n} [e_2]\right) = \mathbf{F}_{\cS}[e_2] \\
&=&\exp\left( \sum_{0<n}\frac{a_n}{2} (p_n^2-p_{2n}) \right), 
\end{eqnarray*}
with~$a_n$ as defined in Proposition~\ref{thm:multinomial}.
\end{theorem}
\begin{proof}
This follows from Propositions~\ref{thm:species} and ~\ref{thm:multinomial}, and
the fact that $e_2=p_i^2-p_2$. 
\end{proof}
Returning to our example:
\begin{equation*}
\label{eq:23Graphs}
\mathbf{G}_{\{2,3\}} = \exp\left( \sum_{0<n}\frac{ (-1)^{n+1}}{2n} \sum_{i=0}^n \binom{n}{i} p_{3n-i}^2-p_{6n-2i}\right).
\end{equation*}

\section{D-finite symmetric series}
\label{sec:DFSymSeries}
Recall that a series $S\in\field[[x_1,\dots,x_n]]$ is 
{\em D-finite\/} in $x_1,\dots,x_n$ when the set of all partial derivatives
and their iterates, $\partial^{i_1+\dots+i_n}F/\partial x_1^{i_1}\dotsm
\partial x_n^{i_n}$,
spans a finite-dimensional vector space over the field
$\field(x_1,\dots, x_n)$.  This was generalized to an infinite number
of variables by Gessel~\cite{Gess90}, who had symmetric functions in
mind. A series $S\in\field[[x_1,x_2,\dotsc]]$ is {\em D-finite\/} in
the $x_i$ if the specialization to 0 of all but a finite (arbritrary)
choice of the variable set results in a D-finite function (in the
finite sense). In this case, many of the properties of the finite
multivariate case hold true. One notable exception is closure under
algebraic substitution, which requires additional hypotheses.

The definition is then tailored to symmetric series by considering the
algebra of symmetric series as generated over~$\mathbb{Q}$ by the set of
power sum symmetric functions~$\{p_1,p_2,\dotsc\}$. A symmetric series
is called {\em D-finite\/} when it is D-finite as a function of the
$p_i$'s. The applicability of this definition will be apparent in a
moment.

The two prototypcial examples, $\mathbf{H}$ and $\mathbf{E}$ are
easily seen to be D-finite, as any such specialization of variables
results in an exponential of a polynomial, which is clearly
D-finite. Similarly, from the expression in
Proposition~\ref{thm:multinomial} we see that the same argument will
hold for any $\mathbf{F}_\cS$.
\begin{theorem} 
\label{thm:Dfinite}
For any set of positive integers~$\cS$, the series
$\mathbf{F}_\cS$ and  $\mathbf{G}_\cS$ are both D-finite symmetric series with
  respect to the $p$-basis.
\end{theorem}
\begin{proof}
  For both of these symmetric series any specialization of the $p$
  variables so that only a finite number are non-zero leaves an
  exponential of a polynomial, which is easily shown to be D-finite in
  the remaining variables. We immediately conclude the D-finiteness of
  both $\mathbf{F}_\cS$ and $\mathbf{G}_\cS$, given the two
  previous results
\end{proof}

\subsection{Extracting the generating functions}
In the power notation for integer partitions,  $\lambda=1^{n_1}\dots k^{n_k}$
indicates that $i$~occurs $n_i$~times in~$\lambda$,
for~$i=1,2,\dots,k$.  The normalization constant
\[
z_\lambda:= 1^{n_1}n_1!\dotsm k^{n_k}n_k!
\]
plays the role of the square of a norm of~$p_\lambda$ in the following important
formula:
\begin{equation}\label{eq:scalp}
\rsc{p_\lambda}{p_\mu}=\delta_{\lambda,\mu}z_\lambda,
\end{equation}
where $\delta_{\lambda,\mu}$~is~1 if $\lambda=\mu$ and 0 otherwise. 

The scalar product is useful for coefficient extraction because
$\rsc{m_\lambda}{h_\mu}=\delta_{\lambda,\mu}$. If we write
$S$ in the form~$\sum_\lambda c_\lambda m_\lambda$,
then the coefficient of $x_1^{\lambda_1}\dots x_k^{\lambda_k}$ in~$S$
is $c_\lambda=\rsc S{h_\lambda}$.

The closure under Hadamard product of D-finite
series~\cite{Lips88} yields the consequence:
\begin{theorem}[Gessel~\cite{Gess90}]\label{thm:rsc_pres_df}
  Let $\mathbf{S}$ and $\mathbf{T}$ be elements of
  $\mathbb{Q}[z][[p_1,p_2,\dotsc]]$, D-finite in the $p_i$'s and $z$,
  and further suppose that $\mathbf{T}$ involves only finitely many of
  the $p_i$'s. Then $\rsc{\mathbf{S}}{\mathbf{T}}$ is D-finite as a
  function of $z$, provided it is well-defined as a power series.
\end{theorem}

From this and Lemma~\ref{thm:Dfinite} our main theorem follows
almost immediately. 
\begin{theorem}\label{thm:Main}
   Let $\cS$ and $\cK$ be sets of positive integers, and
  suppose additionally that~$\cK$ is finite. Let $G_{\cS,\cK}(z)$ be
  the generating function for
  the class~$\mathcal{G}_{\cS,\cK}$ of well-labelled graphs where edge
  weights are from $\cS\cup\{0\}$ and the sum of weights of the edges
  incident to any given vertex is an element of~$\cK$.  Then,
\[G_{\cS,\cK}(z)=\sum_n |\mathcal{G}_{\cS,\cK,n}|\, z^n\] 
  is D-finite. Furthermore, the differential
  equation satisfied by~$G_{\cS,\cK}(z)$ is theoretically computable.
\end{theorem}
\begin{proof}
The series are combinatorial generating functions,
and so they exist.  We remark that
\begin{equation}
\label{eq:scalarproduct}
G_{\cS,\cK}(z) = \rsc {\mathbf{G}_\cS}{\frac{1}{1-\sum_{k\in\cK} h_k z^k}}.
\end{equation}
The first argument to the scalar product is D-finite by
Lemma~\ref{thm:Dfinite}. Furthermore, $\sum_{k\in\cK} h_k z^k$ is a
polynomial in the power sum basis and~$z$ when $\cK$ is finite. Hence
the second argument is rational, and D-finite. The result
follows by Theorem~\ref{thm:rsc_pres_df}. We address the computability
in the next section.
\end{proof}

\begin{example}
The series can be expanded:
\[\begin{split}
\mathbf{G}_{\{3,2\}}=\mathbf{F}_{\{3,2\}}[e_2] 
&=\left(m_{2}+m_{3}+m_{2,2}+m_{3,2}+m_{2,2,2}+m_{3,3}+\dots\right)[e_2]\\
&=m_{2,2}+m_{3,3}+m_{4, 2, 2}+3m_{2, 2, 2, 2}+m_{5, 3, 2}+m_{3, 3, 2,2}\\
&\quad+m_{{6,2,2,2}}+2\,m_{{4,4,2,2}}+6\,m_{{4,2,2,2,2}}+m_{{4,4,4}}+m_{{6,3,
3}}+15\,m_{{2,2,2,2,2,2}}+3\,m_{{3,3,3,3}}+\dots
\end{split}
\]
To extract the generating function of $2$-regular graphs examine the
coefficients of $m_{(2,2,\dots, 2)}$: 
\[G_{\{3,2\}, \{2\}}(z)=\rsc{\mathbf{G}_{\{3,2\}}}{\sum_{0<n}
    h_2^n z^n}=z^2+3z^4+15z^6+105z^8+\dots.\] In this
case, it is corresponds the number of matchings.
Similarly, 
\[G_{\{3,2\}, \{2,3\}}(z)=\rsc{\mathbf{G}_{\{3,2\}}}{\sum_{0<n}\sum_{k=0}^n
    h_3^kh_2^{n-k} z^n}=2z^2+7z^4+36z^6+429z^8+\dots.\]
\end{example}

An extraction for any degree sequence is possible. The D-finiteness
result can be generalized to handle infinite~$\cK$ provided
$\sum_n\sum_{\lambda\vdash n: \lambda_i\in \cK} h_\lambda t^n$ is a
D-finite symmetric series. In general, to determine the generating
function of graphs with a fixed set of degree sequences $\mathcal{D}$,
it suffices to consider the series $\sum_{\mathbf{d}\in \mathcal{D}}  h_dz^{\operatorname{parts}(\mathbf{d})}$.

We could also mark something other than number of vertices. As we
noted, the symmetric series $\sum h_nz^n$ is D-finite, and in fact,
for any finite $k$, the series~$\sum h_n^kz^n$ is D-finite. This would
extract the generating function for the subclass of all regular graphs
on~$k$ vertices from a given graph class with $n$ marking the
regularity. The resulting generating function is also D-finite.

\subsection{Comments on effective computation}
There are two computational tools at hand to compute
$G_{\cS,\cK}(z)$. One could iteratively expand $f_{\cS, n}[e_2]$ in
the monomial basis as we did in the previous example. It might be
slightly more efficient to expand the exponential expression. In
practice, we were able to get some small results with this strategy.

Alternatively, we can make use of the fact that Gessel's result is
effective: given the system of differential equations satisfied by
symmetric series $\mathbf{F}$ and $\mathbf{G}$, there are
algorithms~\cite{ChMiSa05} to compute the differential equation
satisfied by the scalar product. (Of course, at least one of
$\mathbf{F}$ and $\mathbf{G}$ must contain other variables).  It is
straightforward to define the system satisfied by $\mathbf{G}_\cS$,
since it is expressed as an exponential of a polynomial. Consequently
in theory we can compute the differential equation satisfied by
$G_{\cS,\cK}(z)$. In practice, using current algorithms, the
computations are too resource intensive to deliver results when the
number of variables is more than 5. We were able to confirm the
correctness in small cases. 





\section{Other generalizations}
\label{sec:generalK}
By playing with the inner series in the plethysm, we can enumeration
other families of objects, such as hypergraphs, or cyclic
coverings of sets. The details are essentially given in~\cite{Mish07}.

As mentioned above, we could be interested in other kinds of functions
for the growth of $\mathbf{d}$. Other options should be possible to
extract, provided the generating function for the extractor is
D-finite.

A different future direction would be to try to adapt the approach of
de Panafieu and Ramos~\cite{deRa15} for multigraphs to the weighted
edge versions.
\section*{Acknowledgements}
The author gratefully acknowledges funding from NSERC Discovery
Grant. 


\bibliographystyle{acm}
\bibliography{Main}

\end{document}